\documentclass[11pt]{amsart}

\usepackage{amsmath,amssymb,amsthm}
\usepackage{comment}
\usepackage[unicode,breaklinks=true,colorlinks=true]{hyperref}
\usepackage[dvipsnames]{xcolor}

\usepackage[top=1in, bottom=1in, left=1in, right=1in, marginparwidth=1in, marginparsep=0.1in]{geometry}

\numberwithin{equation}{section}
\newtheorem{theorem}{Theorem}[section]
\newtheorem{proposition}[theorem]{Proposition}
\newtheorem{lemma}[theorem]{Lemma}
\newtheorem{definition}[theorem]{Definition}

\theoremstyle{remark}
\newtheorem{remark}[theorem]{Remark}

\definecolor{darkblue}{rgb}{0,0,0.7}

\newcommand{\al}{\alpha}
\newcommand{\be}{\beta}

\newcommand{\e}{\epsilon}
\newcommand{\ga}{{\gamma}}

\newcommand{\la}{\lambda}

\newcommand{\si}{\sigma}
\newcommand{\td}{\tilde}

\newcommand{\De}{\Delta}

\newcommand{\Bp}{\dot B_{p,\infty}^{3/p-1}}

\newcommand{\R}{{\mathbb R }}

\newcommand{\Z}{{\mathbb Z}}

\newcommand{\cN}{{\mathcal N}}

\newcommand{\pd}{{\partial}}
\newcommand{\nb}{{\nabla}}

\newcommand{\I}{\infty}
 
\renewcommand{\div}{\mathop{\mathrm{div}}}

\newcommand{\donothing}[1]{{}}

\newcommand{\EQ}[1]{\begin{equation}\begin{split} #1 \end{split}\end{equation}}

\DeclareMathOperator*{\esssup}{ess\,sup}

\makeatletter
\newcommand{\xRightarrow}[2][]{\ext@arrow 0359\Rightarrowfill@{#1}{#2}}
\makeatother

\newcommand{\loc}{\mathrm{loc}} 
\newcommand{\uloc}{\mathrm{uloc}}

\let\OLDthebibliography\thebibliography
\renewcommand\thebibliography[1]{
  \OLDthebibliography{#1}
  \setlength{\parskip}{1pt}
  \setlength{\itemsep}{1pt plus 0.3ex}
}

\newcommand{\cmild}{{C_{B}}}
\newcommand{\cupper}{{c_1}}
\newcommand{\clower}{{c_2}}

\newcommand{\eLP}{{\epsilon_2}}
\newcommand{\ePhys}{{\epsilon_1}}

\title{
Remarks on the separation of Navier-Stokes flows
}
\author{Zachary Bradshaw}
\date{\today}

\begin{document}

\maketitle

\begin{abstract}
Recently, strong evidence has accumulated that some solutions to the Navier-Stokes equations in physically meaningful classes are not unique.
The primary purpose of this paper is to establish necessary properties for the error of hypothetical non-unique Navier-Stokes flows under conditions motivated by the scaling of the equations.  
Our first set of results show that some scales are necessarily active---comparable in norm to the full error---as solutions separate. 
`Scale' is interpreted in several ways, namely via  algebraic bounds, the Fourier transform and discrete volume elements. These results include a new type of uniqueness criteria which is stated in terms of the error.
The second result is a conditional predictability criteria for the separation of small perturbations. An implication is that the error necessarily  activates at larger scales as flows de-correlate. The last result says that the error of the hypothetical non-unique Leray-Hopf solutions of Jia and \v Sver\' ak locally grows in a self-similar fashion. Consequently, within the Leray-Hopf class, energy can de-correlate at a rate which is faster than linear. This contrasts numerical work on predictability which identifies a linear rate. This discrepancy can likely be explained by the fact that non-uniqueness can be viewed as a perturbation of a singular flow.
\end{abstract}

\section{Introduction}

We consider the Navier-Stokes equations,
\begin{equation}\label{eq.ns}\tag{NS}
\partial_t u - \nu \Delta u +u\cdot \nb u +\nb p= 0;\qquad \nb \cdot u =0,
\end{equation}
which model the motion of a viscous incompressible fluid with velocity $u$ and its associated pressure $p$. We consider the problem on $\R^3\times (0,T)$ for a time $T>0$.
A foundational mathematical treatment of the problem was provided by Leray in \cite{leray} where global weak solutions were constructed for finite energy data.  These solutions are shown to satisfy a global energy inequality, and can therefore be viewed as physically reasonable. Solutions resembling those constructed by Leray are referred to as Leray-Hopf  solutions.  Although it has been nearly a century since Leray's original contribution, important questions remain open about \eqref{eq.ns}. For example, it is not known if Leray-Hopf solutions can possess finite time singularities. It is also unknown if unforced Leray-Hopf solutions  are unique.  In recent years, evidence has accumulated suggesting negative answers to these questions. In the direction of blow-up, Tao has constructed singular solutions for a nonlinear model replicating certain features of \eqref{eq.ns} \cite{Tao}. Regarding uniqueness,  Buckmaster and Vicol have demonstrated non-uniqueness in a class of solutions which is weaker than the Leray-Hopf class using convex integration \cite{BV}. These solutions are not known to satisfy the global or local energy inequalities.  Under non-physical forcing,  non-uniqueness has been shown in the Leray-Hopf class by Albritton, Bru\'e and Colombo \cite{AlBrCo}. In the Leray-Hopf class with no forcing  a conjectural research program of Jia and \v Sver\' ak \cite{JS,JiaSverakIll}, as well as the numerical work of Guillod and \v Sver\' ak \cite{GuSv}, provide strong evidence for non-uniqueness.

The possibility that a deterministic PDE gives rise to multiple solutions is a concern in modeling and forecasting.  Expecting non-uniqueness for \eqref{eq.ns}, it is   important to {understand how non-uniqueness evolves}. Ideally, all possible solutions remain close together, indicating they are predictable from a single flow. This can be viewed as a  sort of stability. The possibility that two solutions can  separate explosively---meaning, e.g.,~that time derivatives of the separation rate are unbounded at $t=0$---is more concerning.  
It is therefore natural to ask:
   \textit{What are necessary properties for the error of non-unique solutions to \eqref{eq.ns}?}
One approach to answering this question would be to establish uniqueness criteria \textit{in terms of the error} by shifting the condition from the background flow to the error.  Importantly, the error can belong to different classes than $u$ or $v$. So, it is feasible that neither $u$ nor $v$ belong to, e.g., a Prodi-Serrin-type space-time Lebesuge space which implies uniqueness, but the error $w=u-v$ does. In the case of non-unique self-similar solutions (in an appropriate class of weak solutions, see, e.g., \cite{BaSeSv}), the error $w$ satisfies $w \in L^\I(0,\I;L^3(\R^3))$. If either $u$ or $v$ belong to this class, then $w\equiv 0$. Therefore, it should not in general be expected that the error fails to be in strong integrability classes. This is good news from the perspective of forecasting; if this were not the case, then non-uniqueness would necessarily imply a large error.

It is not obvious how to write uniqueness criteria in terms of the error.  To prove uniqueness, one typically needs to manage the following sort of bound:
\EQ{\label{ineq.discussion}
[\text{$w$-quantity}] \lesssim   [\text{critical $u$-quantity}]\cdot [\text{$w$-quantity}]. 
}
Classical uniqueness criteria guarantee the critical factor is small. So, the right-hand side is absorbed in the left-hand side and the error consequently vanishes. 
Conditions on $w$ do not help close the estimate. Therefore, uniqueness criteria in terms of the error must look different.

The first goal of this paper is to develop non-uniqueness criteria   in terms of the error. 
Additional  properties of the error will be explored based on the condition that the error separates \textit{en masse} at a scaling invariant rate. In particular,  for several interpretations of ``scale,'' we will show that  intermediate scales which are comparable to $t^{1/2}$ are necessarily active as two solutions separate. 
The second part of this paper examines how small perturbations de-correlate and is motivated by work on predictability \cite{BoMu3D,BoMu2D,PDS}. In this direction, we establish a conditional predictability criteria which states that energy de-correlation requires a certain configuration of activity below wavenumbers in the dissipative range.  Additionally, we  explain how the conjectural existence of non-unique solutions of Jia and \v Sver\' ak \cite{JiaSverakIll} would imply explosive separation within the Leray-Hopf class, indicating that the linear separation rate simulated in \cite{BoMu3D,BoMu2D} is not universal.

\subsection*{Properties of the error of Navier-Stokes flows}
We will use $u$ and $v$ to denote the background flows, with $w=u-v$ being the error.
Before giving the  results, the main assumptions are stated.
The first condition   is that  
\begin{equation}\tag{A1}\label{A1}
\sup_{0<t}  \sup_{3< q\leq \I}  \bigg(   t^{1/2 - 3/(2q)}\big(\|u\|_{L^q} +\|v\|_{L^q}\big)(t)  +   \| u\|_{\dot B^{-1+3/q}_{q,\I}} (t)+ \| v\|_{\dot B^{-1+3/q}_{q,\I}} (t)\bigg) \leq \cupper,
\end{equation}
for some  $\cupper>0$, which can be large. 
In analogy with the blow-up literature, we think of this as a Type I condition in that the controlled quantities are all dimensionless; Type II conditions would only assert weaker bounds.
A centered Type I condition is also natural, namely
\begin{equation}\tag{A1'}\label{A1'}
|u|(x,t) + |v|(x,t) \leq \frac {c_1} {|x|+\sqrt t}.
\end{equation} 
These are motivated by the scaling of \eqref{eq.ns} and the fact that, roughly speaking, if the flows satisfied stronger upper bounds, then they would agree due to the classical uniqueness theories. Thus, these are some of the strictest scenario where non-uniqueness is plausible. Note that \eqref{A1'} implies \eqref{A1}.

The second type of condition essentially says  that  $w$ does not converge to $0$ at a critical rate at  $t=0$, i.e.,
\begin{equation}\tag{A2}\label{A2}
\|w\|_{L^p}(t) t^{1/2-3/(2p)} \geq \clower,
\end{equation} 
for some  $\clower>0$ and some $3<p\leq \I$.
Compared to \eqref{A1} and \eqref{A1'}, it is  less clear that \eqref{A2}   holds in general non-uniqueness scenarios. 
In fact, proving it   does  would amount to shifting a uniqueness criteria from the background flows to the error, which we discussed above. 
When our results use  \eqref{A2}, they should be interpreted as statements about which scales are active given the assumption that   a bulk scaling invariant quantity for the error is non-vanishing at $t=0$.

 We are now ready to state our main results, which apply to three interpretations of ``scale.''
 We define the different solution classes appearing in these theorems in Section 2. Note that for $3<p\leq \I$ there is a local well-posedness theory \cite{Kato,GIM} which implies that any initial data $u_0\in L^p$ produces a unique strong solution  $u$ on $[0,T_0]$ where 
\[
T_0= \td c_p \| u_0\|_{L^p}^{2p/(3-p)},
\]
satisfying 
\[
\sup_{0\leq t\leq T_0}\|u\|_{L^p}(t)\leq 2 \|u_0\|_{L^p},\]
where $\td c_p$ is a constant which only depends on $p$.

 Our first result has in mind a non-uniqueness scenario driven by an isolated singularity at the space-time origin. This is the case for the hypothetical non-uniqueness examples of Jia and \v Sver\' ak \cite{JiaSverakIll}. By \cite{BP1}, these solutions satisfy
 \[
|w(x,t)|\leq \frac {C(u_0) t^{3/2}} {(|x|+\sqrt t)^4}.
 \]
See also \cite{BP2}.
Consequently, for $\ga>0$
\EQ{\label{ineq.similarity-example}
|w(x,t)|\chi_{B^c_{\ga^{-1}\sqrt t}}(x) \lesssim \frac 1 {(\ga^{-1}+1)\sqrt t}. 
}
The prefactor can be made small indicating the activity in the region $|x|\geq \ga^{-1} \sqrt t$ becomes smaller as $\ga^{-1}$ grows. 
On the other hand, by scaling
\[
\| w(\cdot ,t )\|_{L^\I} = \frac 1 {\sqrt t} \|w(\cdot,1)\|_{L^\I}.
\]
Hence, for self-similar non-uniqueness, the majority of the error's activity is necessarily contained within some region $|x|\lesssim \sqrt t$. Our first theorem demonstrates that this property of self-similar uniqueness extends to general scenarios under assumption \eqref{A1'}. It additionally states upper and lower bounds on the error.

\begin{theorem}[Algebraic interpretation of ``scale''] \label{thrm.main}
Assume $u$ and $v$ are mild solutions with the same initial data $u_0$.  
The following hold:  
\begin{enumerate}
    \item (Uniqueness criteria) Fix $3\leq p<\I$. Under assumption \eqref{A1'}, there exist $\ePhys(\cupper,p)>0$ and $\eta = \e_1/(c_1-\e_1)$ so that, if   there exists $T$ so that 
    \EQ{\label{cond.algebraicUniqueness}
    \sup_{0<t<T} \frac {  \|w(x,t)\chi_{B_{\eta^{-1} \sqrt t}} \|_{L^p}}  {  \|w(x,t)\chi_{B^c_{\eta^{-1} \sqrt t}} \|_{L^p}} < \frac \ePhys 2,
    }
    then $w=0$.   
\item (Lower bound on the error)  Let $M_0(t) = 2\max\{ \|u(t)\|_{\I},\|v(t)\|_{\I}\}$.
Assume \eqref{A2} for some non-negative real number   $c_2$. Then, for any $t>0$, $a\geq 0$ and $c_3>0$, there exists  $b=b(c_2,c_3,\td c_\I,a,t,M_0(t))>1$ so that any pair of mild solutions $u$ and $v$ to \eqref{eq.ns} with the same data do \textbf{not} satisfy  
\begin{equation} \label{A3}
 \sup_{x\in \R^3} |w(x,t)| \frac  {(b|x|+\sqrt t)^{a+1}} {\sqrt t^{a}}< c_3.
\end{equation}
\item (Upper bound on the error) Under assumption \eqref{A1'},  
\[
\sup_{x\in \R^3,0<t} |w(x,t)| \frac  {(|x|+\sqrt t)^{4}} {\sqrt{t}^3}\lesssim_{\cupper,C_B}1.
\]
where $C_B$ is a universal constant introduced above \eqref{ineq.bilinear}.
 
\end{enumerate}
 
\end{theorem}

Combining Items (2) and (3), we conclude that, for any $0<a<3$, there exists a sequence $(x_k,t_k)\to (0,0)$   so that 
\[
\frac {\sqrt t^{a}} {c_3(b|x_k|+\sqrt t_k)^{a+1}} \leq |w(x_k,t_k)|. 
\]
Item (1) implies a unique continuation property: 
    If $w=0$ in a neighborhood of $(0,0)$, then $w=0$ globally.
Note that this is automatically satisfied by self-similar solutions due to scaling. It is also implied more generally by real analyticity at positive times.  

To  our knowledge Item (1) constitutes  a new type of uniqueness criteria. We develop similar criteria   below in Theorems \ref{thrm.mainFrequancy} and \ref{thrm.mainDiscrete}.

\bigskip 
 We next consider an analogous result in terms of frequency.
The \textit{a priori} algebraic bounds of self-similar solutions provided  the context for Theorem \ref{thrm.main}.  
The next proposition  provides a similar context in a frequency sense. In particular, it gives  an upper bound on the separation rate of individual Littlewood-Paley frequencies (see Section \ref{sec.prelim} for the definitions of the Littlewood-Paley decompositio, Besov spaces and local energy solutions).

\begin{proposition}[Separation bounds in frequency]\label{prop.frequencyContext}
Assume $u_0\in L^p(\R^3\setminus \{0\})$ for some $p>3$, is divergence free and is $(-1)$-homogeneous. If $u$ and $v$ are self-similar  {local energy solutions}  to \eqref{eq.ns} with the same data $u_0$, then 
\[
\|  \Delta_{< J} w\|_{L^\I}(t)\lesssim 2^{4J} t^{3/2}.
\]
\end{proposition}
This proposition states that, within the self-similar class, the error below a specific frequency vanishes as $t\to 0^+$ and the rate improves at lower scales. Therefore, because the full self-similar error does not vanish (it blows up at the rate $t^{-1/2}$), it must concentrate at smaller and smaller scales as $t\to 0^+$. Also note that the time exponent matches that in Item (3) of Theorem \ref{thrm.main}.
The next theorem establishes properties consistent with this for more general classes of solution.  It is formatted to replicate the structure and themes of Theorem \ref{thrm.main}, but through a different lens. The same comment applies to Theorem \ref{thrm.mainDiscrete} below.

\begin{theorem}[Frequency interpretation of ``scale'']\label{thrm.mainFrequancy} Assume $u$ and $v$ are mild solutions with the same initial data $u_0$. 
The following hold: 
\begin{enumerate}
    \item \text{(Uniqueness criteria)} Assume \eqref{A1} for some value $\cupper$. Fix $p\in  (3,\I]$. There exists $\eLP(\cupper,p)>0$ and $J_1(t)$ with  $2^{J_1(t)}\sim t^{-1/2}$,
  so that, if there exists $T>0$ with 
    \[ \sup_{0<t<T} \frac {\|w_{\geq J_1}\|_{p}}{\|w_{< J_1}\|_p}\leq \eLP,\] then $w=0$. 
\item (Low frequencies are active)  Let $M_0(t)=2\max\{\|u\|_{L^p}(t),\|v\|_{L^p}(t)\}$.
 If $w$  satisfies  \eqref{A2} for some $\clower$ and $p\in (3,\I]$, then 
there exist $\ga(t)$ and $J_2(t)$ with
\[
\ga   =\frac {c_2} {4M_0 {(4\td c_p M_0^{2p/(3-p)}+t)^{1/2-3/(2p)}}},
\]
and  
\[
2^{J_2} \sim \frac  {4M_0 {(4\td c_p M_0^{2p/(3-p)}+t)^{1/2-3/(2p)}}} {c_2}\bigg( \frac {c_2} { 4 M_0^2 C_B (4\td c_p M_0^{2p/(3-p)}+t)^{1/2-3/(2p)}}   \bigg)^{\frac {p} {3-p}},  
\]
so that  we have
\[
\frac{\| w_{\geq J_2}  \|_{L^p} (t)} {\|w_{<J_2}  \|_{L^p}(t)} \leq \ga.
\]
\item (Intermediate frequencies are active) If    \eqref{A1} and \eqref{A2} hold for some $p\in (3,\I]$, then we  have $2^{J_2}\sim \sqrt{t}^{-1}$ and there exists $J_3<J_2$ with $2^{J_3}\sim \sqrt t ^{-1}$ and 
\EQ{ \label{sim.finiteBand}
\| w_{J_3\leq j \leq J_2}\|_p \sim \| w \|_p \sim t^{-1/2+3/(2p)}.
} 
\end{enumerate}
\end{theorem} 

The first item can be interpreted as saying that, if $w\neq 0$, then high modes are active to some extent. This complements the statement of the second item. The same comment applies to Theorem \ref{thrm.mainDiscrete} which appears below.

\bigskip 

Our third iteration of this theme involves a discretized interpretation of ``scale'' which we presently introduce.
Fix a lattice of cubes $\{Q_i\}$ with disjoint interiors, volumes $h^3$, and whose closures cover $\R^3$. Suppose that one cube is centered at the origin.  Let
\[
J_{h} u_0(x) = \sum_j \chi_{Q_j}(x) \frac 1 {|Q_j|} \int_{Q_j} u_0 (y)\,dy.
\]
This is effectively a  discretization of the flow based on volume elements and is a whole-space version of an interpolant operator which has been used extensively to study the number of degrees of freedom in 2D NS flows \cite{FT,FTiti,JonesTiti,JonesTiti2} and more recently in  the Azouni, Olson \& Titi data assimilation paradigm \cite{AOT} and  descendent ideas \cite{CHLMNW}.

\begin{theorem}[Discretized interpretation of ``scale'']\label{thrm.mainDiscrete}
Assume $u$ and $v$ are  mild solutions  with the same initial data.
\begin{enumerate}
    \item \text{(Uniqueness criteria)} Assume \eqref{A1'} for some value $\cupper$ and suppose $u$ and $v$ are $L^{3,\I}$-weak solutions with the same data $u_0$.\footnote{Note that \eqref{A1'} is consistent with the initial data being $O(|x|^{-1})\in L^{3,\I}$, so this is a reasonable class for solutions to belong in. We provide a definition of $L^{3,\I}$-weak solutions in Section 2.} Fix $3<p\leq  \I$.  There exists $\e_3= \e_3(c_1,p,\|u_0\|_{L^{3,\I}})$ so that, letting
	\[
\bar h(t) =  \max\bigg\{ 2 \frac {c_1-\e_3} {\e_3} \sqrt t, \bigg( {t^{3/4}}  \frac { t^{3/(2p)-1/2} } {\e_3 \|w(t)\|_{L^p}  }   \bigg)^{2/3}\bigg\},
	\] 
 if
    \[
    \|w-J_{\bar h}w \|_{L^p}(t) \leq \e_3 \|w\|_{L^p}(t),
    \]
    across a time interval $(0,\delta)$ where $\delta>0$ is arbitrary, 
    then $w=0$.   
    \item \text{(Large scales are active)} Let $M_0(t)=2\max\{\|u\|_{L^p}(t),\|v\|_{L^p}(t)\}$.
    If $w$  satisfies  \eqref{A2} for some $\clower$ and $p\in (3,\I]$, then 
there exist $\ga(t)$ and $h(t)$ with
\[
\ga   =\frac {c_2} {4M_0 {(4\td c_p M_0^{2p/(3-p)}+t)^{1/2-3/(2p)}}},
\]
and  
\[\frac 1 h \sim \frac  {4M_0 {(4\td c_p M_0^{2p/(3-p)}+t)^{1/2-3/(2p)}}} {c_2}\bigg( \frac {c_2} { 4 M_0^2 C_B (4\td c_p M_0^{2p/(3-p)}+t)^{1/2-3/(2p)}}   \bigg)^{\frac {p} {3-p}},  
\]
so that we have 
    \[
    \|J_h w \|_{L^p}(t) \geq\frac \ga 2 \|w\|_{L^p}(t).
    \]
\end{enumerate}
\end{theorem}

Note that, unlike in Theorem \ref{thrm.main} and \ref{thrm.mainFrequancy}, the length scale in the first part depends on $\|w(t)\|_{L^p}$.

\subsection*{Discussion of proofs} Each of Theorem \ref{thrm.main}, \ref{thrm.mainFrequancy} and \ref{thrm.mainDiscrete} are proved using the same basic ideas. For the uniqueness criteria, we choose length scales so that the background flows are necessarily small at large scales. For example,  in the algebraic case and under \eqref{A1'}, if $\sqrt t \ll |x|$, then 
\[
|v(x,t)|\ll \sqrt t^{-1}.
\]
This depletes the large scale activity of the error. The 
smallness condition on the relative size of the small-scales in the error depletes the rest of the error. Of course, the nonlinear nature of the problem means that the large scales of the background flows are not only paired with the large scales of the error in terms like $w\cdot \nb u$, so there is some nuance in pushing these ideas through. 

For the secondary conclusions, we employ an idea worked out in \cite{ABr} which notes that the integral structure of the error,
\[
w(x,t)=e^{t\Delta}w_0 -\int_0^t e^{(t-s)\Delta}\mathbb P\nb \cdot (u\otimes u - v\otimes v)\,ds,
\]
has a part which rapidly decays when   small scales are dominant in the initial error and a part which grows from zero at $t=0$. Hence, if sufficiently small  scales are dominant  at some time, the $L^\I$-norm can be pushed below the assumed lower bound \eqref{A2} at some later time, which is contradictory. This idea has been used in prior work on regularity, in particular it relates to a ``sparseness'' technique  of Gruji\'c \cite{Grujic1} which was re-imagined in \cite{ABr}. 
See also \cite{FGL,BFG,Grujic2,GX}.

\subsection*{Predictability}
Forecasting in turbulent media is possible despite the fact that turbulence is a highly chaotic fluid state because information about different scales persists in the flow for different periods of time. This  can be seen in weather forecasting where the turnover time of the smallest eddies in the atmosphere is on the order of seconds, indicating a rapid onset of chaos at small scales, while weather forecasts are effective for days \cite{BoMu3D}.  In other words,  {large scale effects remain predictable for a non-negligible amount of time despite small scale instabilities.}

Predictability  was initially studied by Lorenz \cite{Lorenz} and Leith and Kraichnan \cite{L,LK}. There is also a rich modern literature, an incomplete list being \cite{BoMu2D,BoMu3D,PDS,TBM}. Several definitions of predictability exist and we adopt that  from \cite{BoMu3D}. For two initial data $u_0$ and $v_0$, define the error energy by $E_\Delta(t) = \| u - v\|_{2}^2(t)$ where $u$ and $v$ evolve from $u_0$ and $v_0$ respectively.  
 For initially small perturbations, the flows are said to be \textit{predictable} if $E_\Delta(t)< \frac \ga 2 (\|u\|_2^2 + \|v\|_2^2)(t)$ for some $\ga\in (0,1)$---for uncorrelated flows the left- and right-hand sides are comparable.
Numerical experiments show that, for infinitesimal perturbations of turbulent flows, on average $E_\Delta(t)$ initially grows exponentially according to $E_\Delta(t) = E_{\Delta}(0 )e^{Lt}$, where $L$ is a Lyapunov exponent, and then settles into a linear growth rate $E_\Delta(t) \sim   t $ \cite{BoMu3D}. 
Linear bounds on growth rates should be expected when perturbing around sufficiently bounded flows. This is even the case when perturbing around an Euler flow, as examined in the context of boundary layer separation by Vasseur and Yang \cite{VY1,VY2}.
The length scale at which the perturbation is given plays a role in the dynamics. A careful description of this can be found in \cite{PDS}.

We include two results which connect predictability to the themes explored in this paper. The first can be viewed as a \textit{conditional predictability criteria} which also sheds light on the distribution of the energy below and above scales in the dissipative range as flows de-correlate.

\begin{theorem}[Predictability criteria]\label{thrm.conditionalPredCrit}
Suppose $u$ and $v$ are distributional solutions to \eqref{eq.ns} on $\R^d\times (0,T)$ for $d=2,3$. 
Suppose also that $w=u-v\in L^\I(0,T;L^2) \cap L^2(0,T;H^1)$ satisfies the energy inequality
\[
\partial_t \|w\|_{L^2}^2 + 2\| \nb w\|_{L^2}^2 \leq  -2\int (w\cdot \nb u)\cdot w \,dx.
\] 
There exists a universal constant $c_5$ so that, if, at a given time there exists $J\in \Z$ so that
\[
\frac{ \|\Delta_{\leq  J} w\|_{L^2}^2 } {\| \Delta_{>J}w \|_{L^2}^2 }\leq    \frac {c_5 2^{2J}}  { \min \{\|u\|_{L^\I} , \sqrt{\|\nb u\|_{L^\I}}\}^2} -1 ,
\]
and
\[
 2^{2J} \geq \frac {\min \{\|u\|_{L^\I} , \sqrt{\|\nb u\|_{L^\I}}\}^2} {c_5},
\]
then 
\EQ{\label{condPredConcl}
\partial_t \| w\|_{L^2}^2  + \frac {C \min \{\|u\|_{L^\I} , \sqrt{\|\nb u\|_{L^\I}}\}^2} {2c_5} \|w\|_{L^2}^2   < 0,
}
meaning the flows are \textit{not} de-correlating at the given time. Similarly, \eqref{condPredConcl} also holds if there exists $h>0$ so that
\[
\frac{ \|J_h w\|_{L^2}^2 } {\|  w-J_hw  \|_{L^2}^2 }\leq   \frac {c_5 h^{-2}}  { \min \{\|u\|_{L^\I} , \sqrt{\|\nb u\|_{L^\I}}\}^2} -1 ,
\]
and
\[
 h^{-2 }\geq \frac {\min \{\|u\|_{L^\I} , \sqrt{\|\nb u\|_{L^\I}}\}^2} {c_5}.
\]

\end{theorem}

Note that the preceding condition is stated at a single time. If the condition holds across the  time interval  $(0,T)$ then $\|w\|_2^2$ is exponentially decaying.  Although the condition is formulated at individual frequencies, the rate of exponential decay is independent of $J$ and $h$. 

We explicitly assume that $w$ has an energy inequality so that we do not need to impose additional conditions on $u$ or $v$. Plainly if $u$ and $v$ are both singular at a particular time, then the conditions of the theorem cannot be met. The problem is symmetric  in $u$ and $v$ so $\min \{\|u\|_{L^\I} , \sqrt{\|\nb u\|_{L^\I}}\}$ can be replaced by $\min \{\|v\|_{L^\I} , \sqrt{\|\nb v\|_{L^\I}}\}$.

Henshaw, Kreiss and Reyna identify a factor of $\|\nb u\|_\I^{-1/2}$ with the dissipative length scale in turbulence \cite{Henshaw}. From this perspective, our result is describing behavior of the error in the dissipative range. If the initial error occurs at very small scales, i.e.~deep within the dissipative range, then due to continuity we must have that the conditions in the theorem are satisfied for a non-vanishing period of time at scales between the perturbation scale and the inertial range. During this time, the error energy would decrease exponentially at a rate which is independent of the scale of the initial perturbation. This would cease once activity builds up at larger scales, at which point the flows would presumably begin to  de-correlate, filling scales in the inertial range in an ``inverse cascade,'' as simulated in \cite{BoMu3D}.   

There are similarities between Theorem \ref{thrm.conditionalPredCrit} and results in data assimilation and determining functionals \cite{AOT,FT,FTiti,JonesTiti,JonesTiti2}, a fact which is visible in our proof.

\bigskip 

Our second result related to predictability explores the universality of the linear separation rate simulated in \cite{BoMu3D} within the class of Leray-Hopf weak solutions. The linear separation rates are for perturbations around bounded flows, but the Leray-Hopf class includes solutions which are unbounded at $t=0$. This would be the case for the localized non-unique solutions \textit{hypothesized} to exist in \cite{JiaSverakIll,GuSv}, where we are viewing non-unique solutions as perturbations with initial perturbation zero.

The hypothetical non-unique Leray-Hopf solutions of Jia and \v Sver\' ak  are built by perturbing\footnote{We are using the term ``perturbing'' a lot. It presently does not refer to the perturbations in the concept of predictability as in the preceding paragraph, but rather to the method by which Jia and \v Sver\' ak generate Leray-Hopf weak solutions from self-similar solutions.} two (hypothetical) self-similar solutions to finite energy solutions. This involves cutting off the tail of the initial data. This should, in principle, not greatly effect the dynamics near the space-time origin, which is where the singularity occurs. In that case the solutions should have an error energy separation which saturates the $t^{1/2}$ rate determined by scaling, as is necessarily the case for the self-similar solutions due to their exact scaling property. This intuition can be made rigorous and suggests that a linear error energy separation rate is likely \textit{not} universal within the Leray-Hopf class.  The following proposition re-states  results from \cite{JiaSverakIll} with the addition of   a new lower bound on the error energy at small times.

\begin{proposition}\label{prop.JS}
Suppose there exists a $(-1)$-homogeneous, divergence free vector field $u_0$ and a self-similar local energy solution $u_1$ satisfying \cite[Spectral Condition (B)]{JiaSverakIll}. Then, there exists a second self-similar local energy solution $u_2$  with the same initial data $u_0$ so that $u_1\neq u_2$. Furthermore, there exist $\td u_0$, $\td u_1$ and $\td u_2$ so that   $u_1+\td u_1$ and $u_2 +\td u_2$  are non-unique Leray-Hopf weak solutions with divergence free initial datum $\td u_0\in L^2$ and, for small enough $t$,
\[
t^{1/2}\lesssim \| (u_1 +\td u_1)-(u_2 +\td u_2)\|_{L^2}^2(t).
\] 
\end{proposition}
Consequently, if these solutions exist, then there is no universal linear bound on the separation rates of perturbations within the Leray-Hopf class.

\subsection*{Organization} Section 2 contains preliminaries including definitions and discussions of mild solutions, local Leray solutions, $L^{3,\I}$-weak solutions and the Littlewood-Paley decomposition. Sections 3, 4 and 5 contain, respectively, the proofs of Theorems \ref{thrm.main}, \ref{thrm.mainFrequancy} and \ref{thrm.mainDiscrete}.  The proofs of results on predictability are respectively contained in Sections 6 and 7.

\section{Preliminaries}\label{sec.prelim}

\subsection{Mild solutions} 
We denoted by $B(\cdot , \cdot)$ the bilinear operator
\[
B(u,v)=- \frac 1 2 \int_0^\tau \underbrace{e^{(\tau-s)\Delta}\mathbb P}_{\text{Oseen tensor}} \nb \cdot ( u\otimes v + v\otimes u  )(s)\,ds,
\]
where $u$ and $v$ are vectors and $\mathbb P$ is the Leray projection operator.
Dumahel's formula applied to the projected form of \eqref{eq.ns} formally leads to the integral representation
\[
u(x,t+\tau)=  e^{\tau\Delta}[u(t)](x)-  \int_0^\tau  {e^{(\tau-s)\Delta}\mathbb P}  \nb \cdot ( u\otimes u )(t+s)\,ds=e^{\tau\Delta}[u(t)](x) +B(u(t+\cdot, u(t+\cdot).
\]
Convergence of the solution to the data is understood in the sense of distributions. 
Although primarily used in the context of strong solutions (in the sense of \cite{Kato}), this formula can be justified rigorously under very general conditions---in particular, it is valid distributionally for many classes of weak solutions \cite{LR,BT7}.  The kernel $K$ of the Oseen tensor
satisfies the following pointwise estimates due to Solonikov \cite{VAS} where $\alpha$ is a multi-index,
\EQ{
\label{ineq.solonikov}
| D^\al_x K(x,t) |\leq C(\alpha) \frac 1 {(|x|+\sqrt t)^{3+|\al| }}.
}
Throughout this paper we will use $C_B$ to denote a universal constant coming from bilinear and kernel estimates. As we will only use the above for $|\alpha|=0,1$, we replace $C(\al)$ by $C_B$ in what follows.  
For $1\leq p\leq \I$,
\EQ{
\label{ineq.bilinear} 
\| B(u,v)\|_{L^p}(t+\tau) \leq \cmild \int_0^\tau \frac 1 {(\tau-s)^{\frac 1 2 +\frac 3 2 (\frac 1 q -\frac 1 p)}} \| u\otimes v (t+s)\|_{L^q}\,ds,
}
where the value of $\cmild$ has been updated. This is from \cite{FJR,Kato} when $1\leq p<\I$ and  \cite[(2.7) and estimates after (3.1)]{GIM} for $p=\I$.
The mild formulation of the perturbed Navier-Stokes equations, where $v$ is the background term and $w$ is the unknown, is
\EQ{ \label{def.wMILD}
w(t+\tau)= e^{\tau \Delta} w(t) - \int_0^\tau e^{(\tau-s)\Delta}\mathbb P\nb \cdot (v\otimes w +w\otimes v+w\otimes w)(t+s)\,ds.
}

\subsection{Local energy solutions}   
We  need to use the properties of local energy solutions in the proof of Proposition \ref{prop.JS}.  These solutions were introduced by Lemari\'e-Rieusset, see the treatments in  \cite{LR,LR2}, and played an important role in the proof of local smoothing in \cite{JS}. 
Because $L^{3,\I}\subset L^2_\uloc$\footnote{$L^2_\uloc$ is the set of all uniformly locally square integrable functions and $L^{3,\I}$ is weak-$L^3$.} it is a natural class in which to consider non-uniqueness \cite{JS,JiaSverakIll,GuSv}. Additional properties of this class have been explored in \cite{KS,KMT,BT8,BT7}.

\begin{definition}[Local energy solutions]\label{def:localEnergy} A vector field $u\in L^2_{\loc}(\R^3\times [0,T))$, $0<T\leq \I$, is a local energy solution to \eqref{eq.ns} with divergence free initial data $u_0\in L^2_{\uloc}(\R^3)$, denoted as $u \in \cN(u_0)$, if:
\begin{enumerate}
\item for some $p\in L^{\frac{3}{2}}_{\loc}(\R^3\times [0,T))$, the pair $(u,p)$ is a distributional solution to \eqref{eq.ns},
\item for any $R>0$, $u$ satisfies
\begin{equation}\notag
\esssup_{0\leq t<R^2\wedge T}\,\sup_{x_0\in \R^3}\, \int_{B_R(x_0 )}\frac 1 2 |u(x,t)|^2\,dx + \sup_{x_0\in \R^3}\int_0^{R^2\wedge T}\int_{B_R(x_0)} |\nb u(x,t)|^2\,dx \,dt<\I,\end{equation}
\item for any $R>0$, $x_0\in \R^3$, and $0<T'< T $, there exists a function of time $c_{x_0,R}\in L^{\frac{3}{2}}_{T'}$  so that, for every $0<t<T'$  and $x \in B_{2R}(x_0)$  
\EQ{ \label{eq:pressure.dec}
p(x,t)&= c_{x_0,R}(t)-\De^{-1}\div \div [(u\otimes u )\chi_{4R} (x-x_0)]
\\&\quad - \int_{\R^3} (K(x-y) - K(x_0 -y)) (u\otimes u)(y,t)(1-\chi_{4R}(y-x_0))\,dy 
,
}
in $L^{\frac{3}{2}}(B_{2R}(x_0)\times (0,T'))$
where  $K(x)$ is the kernel of $\De^{-1}\div \div$,
 $K_{ij}(x) = \pd_i \pd_j \frac {-1}{4\pi|x|}$, and $\chi_{4R} (x)$ is the characteristic function for $B_{4R}$. 
\item for all compact subsets $K$ of $\R^3$,  $u(t)\to u_0$ in $L^2(K)$ as $t\to 0^+$,
\item $u$ is suitable, i.e., for all cylinders $Q\Subset Q_T$ and all non-negative $\phi\in C_c^\I (Q)$, we have  the \emph{local energy inequality}
\EQ{\label{ineq:CKN-LEI}
2\iint |\nb u|^2\phi\,dx\,dt \leq \iint |u|^2(\pd_t \phi + \De\phi )\,dx\,dt +\iint (|u|^2+2p)(u\cdot \nb\phi)\,dx\,dt,
}
\item the function
\EQ{\label{cont}
t\mapsto \int_{\R^3} u(x,t)\cdot {w(x)}\,dx,
}
is continuous in $t\in [0,T)$, for any compactly supported $w\in L^2(\R^3)$.
\end{enumerate}
\end{definition}

Local energy solutions are known to satisfy certain \textit{a priori} bounds \cite{LR}. For example, in \cite{JS,BT8}, the following \textit{a priori} bound is proven: 
Let $u_0\in L^2_\uloc$, $\div u_0=0$, and assume $u\in \mathcal N (u_0)$.  For all $r>0$ we have
\begin{equation}\label{ineq.apriorilocal}
\esssup_{0\leq t \leq \sigma r^2}\sup_{x_0\in \R^3} \int_{B_r(x_0)}\frac {|u|^2} 2 \,dx\,dt + \sup_{x_0\in \R^3}\int_0^{\sigma r^2}\int_{B_r(x_0)} |\nabla u|^2\,dx\,dt <CA_0(r) ,
\end{equation}
where
\[
A_0(r)=rN^0_r= \sup_{x_0\in \R^3} \int_{B_r(x_0)} |u_0|^2 \,dx,
\] 
and
\begin{equation}\label{def.sigma}
\si=\sigma(r) =c_0\, \min\big\{(N^0_r)^{-2} , 1  \big\},
\end{equation}
for a small universal constant $c_0>0$.  Additionally, local energy solutions are mild \cite{BT7}.

\subsection{$L^{3,\I}$-weak solutions}
Local energy solutions are defined for initial data in $L^2_\uloc$. Note that $L^{3,\I}$ embeds in $L^2_\uloc$. Thus, when $u_0\in L^{3,\I}$, a local energy solution exists. The scaling of $L^{3,\I}$ does not, however, show up in the properties of this solution which come from the definition of local energy solutions. 
The class of  $L^{3,\I}$-weak solutions provides a notion of solution which is more tailored to the scaling of $L^{3,\I}$. This class was introduced by Barker, Seregin and \v Sver\' ak in \cite{BaSeSv} and extends ideas in \cite{SeSv}. It has since been extended to non-endpoint critical Besov spaces of negative smoothness \cite{AB}. 

\begin{definition}[Weak $L^{3,\I}$-solutions] Let $T>0$ be finite. Assume $u_0\in L^{3,\I}$ is divergence free. We say that $u$ and an associated pressure $p$ comprise a weak $L^{3,\I}$-solution if \begin{enumerate}
 \item $(u,p)$ satisfies \eqref{eq.ns} distributionally,  \eqref{cont} and   the local energy inequality \eqref{ineq:CKN-LEI},
 \item $\td u :=u-e^{t\Delta}u_0$ satisfies, for all $t\in (0,T)$,
 \EQ{\label{ineq.BSSbound}
 \sup_{0<s<t}\| \td u \|^2_{L^2} (s) + \int_0^t \| \nb \td u\|_{L^2}^2(s)\,ds   <\I,
 }
 and
\EQ{ \label{ineq:energyIneq}
 \| \td u\|_{L^2}^2(t) +2\int_0^t \int|\nb \td u|^2\,dx\,ds\leq 2\int_0^t \int ( e^{s\Delta}u_0 \otimes \td u + e^{s\Delta}u_0 \otimes e^{s\Delta}u_0) : \nb \td u\,dx\,ds.
 }
\end{enumerate}
\end{definition}
In \cite{BaSeSv}, weak solutions are constructed which satisfy the above definition for all $T>0$. 
 Also, due to their spatial decay, weak $L^{3,\I}$-solutions are mild and, in view of \cite{BT7}, are local energy solutions.

An important observation in \cite{BaSeSv} is that the nonlinear part of a weak $L^{3,\I}$-solution satisfies a dimensionless energy estimate, namely
\EQ{\label{ineq.BSSdecay}
 \sup_{0<s<t}\| \td u \|_{L^2} (s) +\bigg(\int_0^t \| \nb \td u\|_{L^2}^2(s)\,ds \bigg)^\frac 1 2 \lesssim_{u_0} t^{\frac 1 4}.
}
We emphasize that the energy associated with $\td u$ vanishes at $t=0$. This decay property will be essential in our work.

\subsection{Littlewood-Paley} We refer the reader to \cite{BCD} for an in-depth treatment of Littlewood-Paley and Besov spaces. Let $\lambda_j=2^j$ be an inverse length and let $B_r$ denote the ball of radius $r$ centered at the origin.  Fix a non-negative, radial cut-off function $\chi\in C_0^\infty(B_{1})$ so that $\chi(\xi)=1$ for all $\xi\in B_{1/2}$. Let $\phi(\xi)=\chi(\lambda_1^{-1}\xi)-\chi(\xi)$ and $\phi_j(\xi)=\phi(\lambda_j^{-1})(\xi)$.  Suppose that $u$ is a vector field of tempered distributions and let $\Delta_j u=\mathcal F^{-1}\phi_j*u$ for $j\geq 0$ and $\Delta_{-1}=\mathcal F^{-1}\chi*u$. Then, $u$ can be written as\[u=\sum_{j\geq -1}\Delta_j u.\]
If $\mathcal F^{-1}\phi_j*u\to 0$ as $j\to -\infty$ in the space of tempered distributions, then we define $\dot \Delta_j u = \mathcal F^{-1}\phi_j*u$ and have
\[u=\sum_{j\in \Z}\dot \Delta_j u.\]
We additionally define
\[
\Delta_{<J} f = \sum_{j<J} \dot \Delta _jf;\quad \Delta_{\geq j} f =f- \Delta_{<J} f,
\]
with the obvious modifications for $\Delta_{\leq J}$ and $\Delta_{>J}$. If we do not specify that $J$ is in integer, then we use $\chi(\la_1^{-1} 2^{J} \xi)$ in the definition of $\Delta_{\leq J}$.

Littlewood-Paley blocks interact nicely with derivatives and, by Young's inequality, $L^p$ norms. This is   illustrated by the Bernstein inequalities which read:
\[
\| D^\al \dot \Delta _jf \|_{L^p} \leq   2^{j|\al| } \|\dot \Delta _jf\|_{L^p}; \quad   \|   \dot \Delta _jf \|_{L^p} \leq   2^{j (\frac 3 q - \frac 3 p) } \|\dot \Delta _jf \|_{L^q}. 
\]

The Littlewood-Paley formalism is commonly used to define Besov spaces. 
We are primarily interested in Besov spaces with infinite summability index, the norms of which are
\begin{align*}
&||u||_{B^s_{p,\infty}}:= \sup_{-1\leq j<\infty } \lambda_j^s ||\Delta_j u ||_{L^p(\R^n)},
\end{align*}
and
\begin{align*}
&||u||_{\dot B^s_{p,\infty}}:= \sup_{-\infty< j<\infty } \lambda_j^s ||\dot \Delta_j u ||_{L^p(\R^n)}.
\end{align*} 
The critical scale of Besov spaces are $\Bp$. Note that $L^3\subset L^{3,\I} \subset \Bp$ for $3<p$. In particular, $\Bp$ contains functions $f$ satisfying $|f(x)|\lesssim |x|^{-1}$ when $p>3$.

\section{Algebraic  scenario} 

The following definition and lemma appear in \cite{AB}. They will be used to prove Item (2) of Theorem \ref{thrm.main}. 
\begin{definition}[$L^p$-sparseness]
\label{def:lpsparseness}
Let $1 \leq p \leq \infty$, $ \varepsilon,  \be\in (0,1)$, and $ \ell>0$. A vector field $u_0 \in L^p(\R^d)$ is \emph{$(\varepsilon,\be,\ell)$-sparse in $L^p$} if there exists a measurable set $S$ such that
\begin{equation}
    \|u_0\|_{L^p (S^c)} <    \beta  \|u_0\|_{L^p}
\end{equation}
and 
\begin{equation}
\sup_{x_0\in \R^d} \frac {|S\cap B_{  \ell }(x_0) |} {|B_{ \ell}(x_0) |} \leq \varepsilon.    
\end{equation}
\end{definition} 

Let $G : \R^d \to \R$ be a Schwartz function and $G_t$ be the convolution operator
\begin{equation}
    G_t u_0 := t^{-\frac{d}{2}} G(\cdot/\sqrt{t}) \ast u_0.
\end{equation}
when $t > 0$. We have in mind that $G = (4\pi)^{-d/2} e^{-|x|^2/4}$ and $G_t$ is the heat semigroup.

\begin{lemma}\label{lemma.heat}
 Let $p \in (1,\infty]$, $\gamma \in (0,1)$, and $t>0$ be fixed. Let $u_0\in L^p(\R^d)$ be a vector field. Suppose that $u_0$ is $(\varepsilon,\be,\bar \ell \sqrt {t})$-sparse, where the dimensionless parameters $\varepsilon, \be \in (0,1)$ and $\bar\ell > 0$ satisfy
 \begin{equation}
    \label{eq:therequirementsforsparsitytouselater}
    \bar \ell \geq f(\gamma);\quad  {\beta} \leq \| G \|_{L^1}^{-1} \gamma / 3;\quad  {\varepsilon^{1-\frac{1}{p}}} \leq C_0^{-1} \| G \|_{L^\infty}^{-1} \gamma / \bar{\ell}^d
\end{equation}
where $f$ depends on $G$ and satisfies $f(\gamma) \to +\infty$ as $\gamma \to 0^+$, and
$C_0 > 1$ is an absolute constant depending only on the dimension.  Then \begin{equation}\label{ineq.caloric.decay}
\|  G_t u_0\|_{L^p} \leq \gamma \|u_0\|_{L^p}.    
\end{equation}
When $G_t = e^{t \Delta}$, the above requirement on $\bar{\ell}$ can be made more explicit:
 \begin{equation}
    \label{eq:therequirementsforlesparsity}
    \bar \ell^2 \geq C_0 {\ln(C_0/\gamma)}.
\end{equation}
\end{lemma}

\begin{proof}[Proof of Theorem \ref{thrm.main}]

(Part 1)  
Note that $\sup_{0<s<t} s^{(p-3)/2p}\| w\|_{L^p}(s)<\I$ by \eqref{A1'}.  We will shortly specify a value for $\ePhys$. 
Fix $3<p< \I$.
From \eqref{def.wMILD} with $\tau=0$,   we have
\EQ{
\| w(x,t)\|_{L^p} \leq C_B \int_0^t \frac 1 {(t-s)^{1/2}} \big( \| v  w \|_{L^p}  +\|uw\|_{L^p}   \big) \,ds.
}
We choose $\eta$ to be  
\EQ{\label{def.gamma1}
\eta =\frac {\e_1} {c_1-\e_1}.
}
If $|y|\geq \eta^{-1}\sqrt s$, then, by \eqref{A1'} and our requirement on $\eta$,
\[
|v(y,s)|\leq \frac {\e_1} {\sqrt s}.
\]
Hence, 
\[
\| v w\|_{L^p(|y|\geq \eta^{-1}\sqrt s)}  \leq \frac {\e_1} {\sqrt s} \|w\|_{L^p(|y|\geq \eta^{-1}\sqrt s)}.
\]
Assume $T$ is small enough that
\[\frac {  \|w(x,t)\chi_{B_{\eta^{-1} \sqrt t}} \|_{L^p}}  {  \|w(x,t)\chi_{B^c_{\eta^{-1} \sqrt t}} \|_{L^p}} < \ePhys,
\]
for all $0<t<T$.
Using \eqref{cond.algebraicUniqueness}, 
\[
\| v w\|_{L^p(|y|< \eta^{-1}\sqrt s)}  \leq \|v\|_{L^\I} \| w\|_{L^p(|y|< \eta^{-1}\sqrt s)}  \leq \e_1 \frac {c_1} {\sqrt s} \| w\|_{L^p(|y|\geq  \eta^{-1}\sqrt s)} .
\]

The same estimates hold with $v$ replaced by $u$.
Hence, for $t<T$
\EQ{
\| w(x,t)\|_{L^p} &\leq  C_B \int_0^t \frac {(1+c_1)\e_1} {(t-s)^{1/2}s^{1/2+(p-3)/2p}} s^{(p-3)/2p}\| w\|_{L^p}(s) \,ds 
\\&\leq C_B (1+c_1) \e_1 t^{(3-p)/2p} \sup_{0<s<t} s^{(p-3)/2p}\| w\|_{L^p}(s) .
}
Taking 
\[
\ePhys \leq \frac 1 {2 C_B(1+c_1)},
\]
and, after rearranging things, a supremum on the left-hand side,
we have 
\[
\sup_{0<s<T} s^{(p-3)/2p}\| w\|_{L^p}(s)   \leq 0.
\]
In light of \eqref{cond.algebraicUniqueness} we conclude that $w=0$ on $\R^3\times (0,T)$. Global uniqueness follows from the local well-posedness theory and the fact that $u(\cdot,T/2) \in L^\I (\R)$.

If $p=\I$, then we set up our argument slightly differently,\footnote{To illustrate the details which are omitted here, we pursue this case in the proof of Theorem \ref{thrm.mainFrequancy} below, as it requires similar logic.} beginning with,
\EQ{
\| w(x,t)\|_{L^\I} \leq C_B \int_0^t \frac 1 {(t-s)^{1/2+ 3/(2q)}} \|w\|_{L^\I} \big( \| v   \|_{L^q}  +\|u\|_{L^q}   \big) \,ds,
}
and then reason similarly. Ultimately this avoids having to integrate
\[
\int_0^t \frac 1 {(t-s)^{1/2}s}\,ds, 
\]
which diverges.

\bigskip \noindent (Part 2)  We assume \eqref{A2}.
Fix $t_0$, which plays the role of $t$ in the statement of the theorem. Let $M_0 = 2 \max \{  \| u \|_{L^\I}(t_0),\|v\|_{L^\I}(t_0)\}$. 
Then, there exists $T_0= 4\td c_\I M_0^{-2}$ so that $\| u \|_{L^\I}(t)+ \|v\|_{L^\I}(t) \leq 2(\| u \|_{L^\I}(t_0)+ \|v\|_{L^\I}(t_0))$ for all $t_0\leq t\leq t_0+T_0$.
Bilinear estimates imply
\[
\| w(t) \|_{L^\I}\leq \| e^{(t-t_0)\Delta} w(t_0)\|_{L^\I} + 2C_B(t-t_0)^{1/2}  M_0^2.
\]
We have 
\[
2C_B(t-t_0)^{1/2}  M_0^2 \leq \frac {c_2} {2\sqrt t},
\]
if
\[
2C_B(t-t_0)^{1/2}  M_0^2 \leq \frac {c_2} {2\sqrt{T_0+t_0}}.
\]
Choose $t$ to satisfy
\EQ{\label{breakpoint}
t=t_0 + \frac {c_2^2} {16C_B^2 (\td c_\I M_0^2+t_0 M_0^4)}.
}
 
We will show that, if \eqref{A3} holds at time $t_0$ for a choice of parameters to be specified momentarily, then 
\[
\| e^{(t-t_0)\Delta}w(t_0)\|_{L^\I} \leq \frac {c_2} {4 \sqrt t},
\]
leading to a contradiction, namely
\[
\|w(t)\|_{L^\I}\leq \frac {3 c_2 } {4\sqrt t},
\]
as this violates \eqref{A2}.
For this we will use Lemma \ref{lemma.heat} with 
\[
\ga = \frac {c_2} {4 M_0  \sqrt t}.
\]
Define $\bar \ell$, $\be$ and $\e$ according to this choice of $\ga$, the time scale $t-t_0$ and \eqref{eq:therequirementsforlesparsity}. 
If \eqref{A3} holds at time $t_0$, then,
\EQ{
S_t := \{ x\in \R^3: |w(x,t_0)|\geq \be \|w\|_{L^\I}(t_0) \} 
&\subset \bigg\{\frac {2c_3\sqrt {t_0}^a} {(b|x|+\sqrt {t_0})^{a+1}} \geq \be \frac {c_2} {\sqrt {t_0}}\bigg\}
\\&\subset \bigg\{  \bigg(  \frac {2c_3} {\be c_2} \bigg)^{1/(a+1)}   \sqrt {t_0 }    \geq b|x|  \bigg\}.
}  
To ensure $(\epsilon, \be , \bar \ell \sqrt{t-t_0})$-sparseness we choose $b$ so that 
\[
\bigg| \bigg\{  \bigg(  \frac {2c_3} {\be c_2} \bigg)^{1/(a+1)}   \sqrt{t_0}     \geq b|x|  \bigg\} \bigg| \leq \e \big|B_{\bar \ell \sqrt{t-t_0}}\big|,
\]  
namely,
\[
\frac 1 {b^3}:= \frac {\epsilon \bar \ell^3 (t-t_0)^{3/2}} {(2c_3 /(\be c_2))^{3/(a+1)} t_0^{3/2}}. 
\]
Under this choice of parameters, by Lemma \ref{lemma.heat},
\[
\|e^{(t-t_0)\Delta}w(t_0)\|_{L^\I} \leq \frac {c_2} {4\sqrt t},
\]
which we already noted is a contradiction. Therefore, we cannot have that \eqref{A3} holds at any time $t_0$.

\bigskip \noindent (Part 3) Part 3 follows from an argument appearing in \cite{JiaSverakIll} and again  in \cite{Tsai-DSSI,Tsai-book,BP1}. We will need point-wise bounds for some convolutions which we copy from  \cite[Lemma 2.1]{Tsai-DSSI}: Let $a,b\in(0,5)$  and $a+b>3$. Then, 
\EQ{\phi(x,a,b) = \int_0^1 \int_{\R^3} (|x-y|+\sqrt{1-t})^{-a}(|y|+\sqrt{t})^{-b} \,dy\,dt,} 
is well defined for $x\in \R^3$, and
\EQ{\label{ineq:Tsai.integral}\phi(x,a,b) \lesssim R^{-a} + R^{-b} + R^{3-a-b} [ 1+ (1_{a=3}+1_{b=3})\log R],}
where $R=|x|+2$. 
 These estimates can be extended to  {other time intervals} by a change of variable. From \eqref{ineq.solonikov} we have
 \[
|w(x,t)|\leq \int_0^t \int \frac {C_B} {(|x-y|+\sqrt{t-s})^4}\frac {c_1^2} {(|y|+\sqrt s )^2}\,ds,
 \]
 which implies 
\[
|w(x,t)| \lesssim \frac {\sqrt t } {(|x|+\sqrt t)^2},
\]
by the preceding convolution estimates.
Applying this argument two more times yields the advertised result.  We illustrate the first application. We have using the new bound for $|w|$ and \eqref{A1'} that 
\[
B(u,w) \lesssim \sqrt t \int_0^t \int \frac {C_B} {(|x-y|+\sqrt{t-s})^4}|w| |u|\,ds \lesssim  \sqrt t\int_0^t \int \frac {C_B} {(|x-y|+\sqrt{t-s})^4}\frac 1 {(|y|+\sqrt s)^3}\,ds.
\]
The pointwise estimates for the time convolution imply this term is bounded by $ t/(|x|+\sqrt t)^3.$ The same applies to the other terms in the integral expansion for $w$. One more repetition results in the bound $t^{3/2}/(|x|+\sqrt t)^4$.
\end{proof}

\section{Frequency  scenario}

We begin by proving separation rates for individual modes of self-similar solutions.  

 \begin{proof}[Proof of Proposition \ref{prop.frequencyContext}]
Observe that
\[
\frac 1 {|x|^n} \in \dot B^{-n}_{\I,\I}.
\]
This follows from the fact that $|x|^{-n}$ is $(-n)$-homogeneous and the relationship this induces   on $\| \dot\Delta_j (|\cdot|^{-n})\|_{L^\I}$ compared to $\| \dot\Delta_0 (|\cdot|^{-n})\|_{L^\I}$.  
In \cite{BP1} it is shown that $w(x,1)\lesssim (1+|x|)^4$ (this is why we require the solutions be local energy solutions). Hence, $w(\cdot, 1)\in \dot B^{-4}_{\I,\I}$. Additionally, the discrete self-similar relationship between modes can be calculated as in \cite[(2.5)]{BT3}, implying, for a given $j$ and $t$ and letting $2^{2(j-i)}t=1$,
\[
\| \dot \Delta_{j} w(\cdot, t)\|_{L^\I} =  \frac 1 {\sqrt t} \| \dot \Delta_i w(2^{j-i}x,1)\|_{L^\I}\lesssim  \frac 1 {\sqrt t} 2^{4i},
\]
due to membership in $\dot B^{-4}_{\I,\I}$. Hence,
\[
\| \dot \Delta_{j} w(\cdot, t)\|_{L^\I}  \lesssim 2^{4j} t^{3/2}.
\]
The stated conclusion follows after summing over $j<J$.
\end{proof}

Note that we used a global bound on the profile at $t=1$. For discretely self-similar solutions, these bounds are only available away from the origin---it is not known in general whether or not the solutions can be singular on a ball centered at the origin. Additional work would therefore be needed to check that $w(\cdot, 1)\in \dot B^{-4}_{\I,\I}$.

\bigskip 
We now recall a definition of sparseness framed in terms of the Littlewood-Paley decomposition. Compared to physical sparseness, this definition has the advantage of involving fewer parameters. This notion of sparseness  encompasses the spatial version, at least within a certain parameter range, as demonstrated in \cite{ABr}.

\begin{definition}[$L^p$-sparseness in frequency]
\label{def:frequencysparseness}
Let $\be \in (0,1)$ and $J\in \R$. Then a vector field $u_0 \in L^p$ is \emph{$(\be,J)$-sparse in frequency in $L^p$} if 
\begin{equation}
\| \Delta_{<J}u_0\|_{L^p}\leq \be \|u_0\|_{L^p}.    
\end{equation}
\end{definition}

 The following lemma is taken from \cite{ABr}.

\begin{lemma}\label{lemma.heat.frequency}Fix $1\leq p\leq \I$, $t>0$ and $\ga>0$. Let $u_0\in L^p$. There exists $J\in \Z$ satisfying $2^{J}\sim \ga^{-1} t^{-1/2}$ and $\be=\ga/2$ so that, if $u_0$ is $(\be,J)$-sparse in frequency, then  
\begin{equation}
\| e^{t\Delta} u_0 \|_{L^{p}} \leq \ga \|u_0\|_{L^p}.    
\end{equation}
\end{lemma}

 We our now ready to prove Theorem \ref{thrm.mainFrequancy}.
 
\begin{proof}[Proof of Theorem \ref{thrm.mainFrequancy}]

\medskip 
\noindent  {(Part 1)} We first prove the $p=\I$ case.  
Applying $\dot \Delta_{<{J_1}}$ to the perturbed Navier-Stokes equations and adopting the abbreviation 
$f_{<J} = \dot \Delta_{<J}$ (as well as similar abbreviations for $\dot \Delta_{\leq J}$, $\dot \Delta_{\geq J}$ and $\dot \Delta_{> J}$)
gives
\[
(\partial_t - \Delta ) w_{< {J_1}} +  \mathbb P \nb \cdot (w\otimes 
 v +v\otimes w 
 + w\otimes 
 w)_{<{J_1}}=0.
\]
Then, 
\[
w_{< {J_1}}(x,t ) =   -\int_0^t e^{(t-s)\Delta} \mathbb P \nb \cdot (w\otimes 
 v +v\otimes w 
 + w\otimes 
 w)_{< {J_1}}(s)\,ds.
\]
Note that, provided $v\in L^\I(0,T;\dot B^{-1}_{\I,\I})$, we have 
\[
\|v_{< {{j}}}\|_{L^\I} \leq \|v\|_{L^\I(0,T;\dot B^{-1}_{\I,\I})} 2^{{j}}.
\]
Also, by support considerations in the Fourier variable, we have
\[
(f_{<{J_1}} \ g)_{<{J_1}}  = (f_{<{J_1}}\  g_{<{J_1}+2})_{<{J_1}}.
\]
The bilinear terms are all bounded the same, as is illustrated in the following estimate where $q$ is taken in $(3,\I)$,
\EQ{
& \bigg| \int_0^t e^{(t-s)\Delta} \mathbb P \nb \cdot (w\otimes 
 v )_{<{J_1}}(s)\,ds\bigg| 
 \\&\leq \cmild  \int_0^t \frac 1 {(t-s)^{1/2}} \|w_{<{J_1}} \otimes v_{< {J_1}+2}
 \|_{L^\I}\,ds + C_B\int_0^t \frac 1 {(t-s)^{\frac 1 2 + \frac 3{2q}}}\|(w_{\geq  {J_1}} \otimes v)_{< {J_1}}\|_{L^q}\,ds
\\&\leq \cmild \cupper   2^{{J_1}+2}\sup_{0<s<t}s^{1/2} \|w_{< {J_1}}\|_\I(s) +\cmild   \int_0^t \frac {\cupper} {(t-s)^{\frac 1 2 +\frac 3 {2 q}} s^{1- \frac 3 {2q}  }}   s^{\frac 1 2} \|w_{\geq {J_1}}\|_{L^\I}(s)\,ds
\\&\leq (\cmild \cupper   2^{{J_1}+2}  + \cmild \cupper \eLP   t^{-1/2} )\sup_{0\leq s\leq t}s^{1/2} \|w_{< {J_1}}\|_\I(s),
}
which holds by \eqref{A1} and \eqref{ineq.bilinear}.
We therefore take $2^{J_1} = 2^{{J_1}(t)}$  and $\eLP$ to satisfy 
\[
\cmild \cupper   2^{{J_1}+2}  + \cmild \cupper \eLP   t^{-1/2} < \frac 1 3 t^{-1/2},
\]
whence obtaining
\[
 t^{1/2} \bigg\| \int_0^t e^{(t-s)\Delta} \mathbb P \nb \cdot (w\otimes 
 v )(s)\,ds\bigg\|_{L^\I} < \sup_{0\leq s\leq t}\frac 1 3 s^{1/2}\|w_{<{J_1}}(s)\|_{L^\I}.
\]
Repeating this for the other terms in the expansion for $w_{< {J_1}}$ and  taking a time-supremum of the left-hand side of the $w_{<{J_1}}$ integral expansion implies  $w=0$.

The $3<p<\I$ case is  similar but we do not need to pass to the $L^q$ norm in our bilinear estimate---this resembles what we did in detail in the proof of Theorem \ref{thrm.main}.

 \medskip 
\noindent  {(Part 2)}    
We assume \eqref{A2}.
Fix $t_0$. Let $M_0 = 2 \max \{  \| u \|_{L^p}(t_0),\|v\|_{L^p}(t_0)\}$. 
Then, there exists $T_0= 4\td c_p M_0^{2p/(3-p)}$ so that $\| u \|_{L^p}(t)+ \|v\|_{L^p}(t) \leq 2(\| u \|_{L^p}(t_0)+ \|v\|_{L^p}(t_0))$ for all $t_0\leq t\leq t_0+T_0$.
Bilinear estimates imply
\[
\| w(t) \|_{L^p}\leq \| e^{(t-t_0)\Delta} w(t_0)\|_{L^p} + 2C_B(t-t_0)^{1/2-3/(2p)}  M_0^2.
\]
We have 
\[
2C_B(t-t_0)^{1/2-3/(2p)}  M_0^2 \leq \frac {c_2} {2t^{1/2-3/(2p)}},
\]
if
\[
2C_B(t-t_0)^{1/2-3/(2p)}  M_0^2 \leq \frac {c_2} {2(T_0+t_0)^{1/2-3/(2p)}}.
\]
Choose $t$ to satisfy
\EQ{ 
t=t_0 + \bigg( \frac {c_2} { 4 M_0^2 C_B (T_0+t_0)^{1/2-3/(2p)}}   \bigg)^{\frac {2p} {p-3}} 
}

We will  obtain the contradiction if 
\[
\| e^{(t-t_0)\Delta}w(t_0)\|_{L^p} \leq \frac {c_2} {4 t^{1/2-3/(2p)}}. 
\]
Lemma \ref{lemma.heat.frequency} will give us a conclusion like
\[
\| e^{(t-t_0)\Delta}w(t_0)\|_{L^p}  \leq \ga \|w(t_0)\|_{L^p},
\]
which is bounded above by $\ga M_0$. We therefore want to use Lemma \ref{lemma.heat.frequency}  with $\ga$ satisfying
\[
\ga M_0 \leq \frac {c_2} {4 {(T_0+t_0)^{1/2-3/(2p)}}}, \text{ and so we fix }\ga   =\frac {c_2} {4M_0 {(T_0+t_0)^{1/2-3/(2p)}}}.
\]
Choosing $\beta = \ga /2$ 
and, recalling $t-t_0$ is given in \eqref{breakpoint}, choosing
\[
2^{J_2} \sim \ga^{-1}(t-t_0)^{-1/2}, 
\]
we see by Lemma \ref{lemma.heat.frequency} that if  $w(t_0)$ is $(\beta,J_2)$-sparse in $L^p$,
then 
\[
\|e^{(t-t_0)\Delta}w(t_0)\|_{L^p} \leq \frac {c_2} {4 t^{1/2-3/(2p)}}.
\]
Therefore, $w(t_0)$ can never be $(\beta,J_2)$-sparse in $L^p$.

\medskip 
\noindent {(Part 3)}  Note that \eqref{A1} and \eqref{A2} together imply
\[
\| w_{\leq J_2}{(t)}\|_{L^p} \sim t^{3/(2p)-1/2} \sim \|w(t)\|_{L^p}.
\]
By \eqref{A1}, $w\in L^\I_t \dot B^{-1}_{\I,\I}$. So, 
$\| w_{< J_3}\|_\I \leq 2 \cupper 2^{( J_3+1)(1-3/p)}$.
We  require that
\[
2 \cupper 2^{( J_3+1)(1-3/p)}\lesssim  t^{3/(2p)-1/2} ,
\]
where the suppressed constant is chosen small enough that
\[
\| w_{< J_3}\|_p \leq \frac 1 2 \| w_{\leq  J_2}\|_{L^p}.
\]
It follows that 
\[
\| w_{J_3 \leq j \leq J_2} \|_{L^p} \geq \frac 1 2 \|w_{\leq J_2}\|_{L^p},
\]
which completes the proof.

\end{proof}

\begin{remark}It may be interesting to point out that, under the assumptions of Theorem \ref{thrm.mainFrequancy}, if \eqref{A1} and \eqref{A2} hold, the latter for $p=\I$,  then there exists $\e_4$ so that,   
\[
\inf_{t>0} \|w\|_{\dot B^{-1}_{\I,\I}}(t) > \e_4.
\] 
In other words, a condition like \eqref{A2}  implies other scaling invariant measurements of the error do not vanish at $t=0$.
To prove this, recall \eqref{def.wMILD} and observe that, for $t>0$,
\EQ{ \label{ineq.mix}
\int_0^\tau e^{(\tau-s)\Delta}\mathbb P\nb \cdot (u\otimes w )(t+s)\,ds
&\leq \cmild
\sup_{t<s' <t+\tau }   \int_0^{\tau} \frac 1 {(\tau-s)^{1/2}} \frac {\cupper^2} {t } \,ds
\\&\leq \frac {\cmild \cupper^2 \tau^{1/2}} {t} 
\\&\leq \frac {c_2} {6 \sqrt{t+\tau}},
}
provided 
\[
\tau =  \frac {c_2^2 t} {72 C_B^2 c_1^4 }.
\]
Other terms in the bilinear part of \eqref{def.wMILD} are handled identically. Therefore, 
\[
\frac \clower {(t+\tau)^{1/2}} \leq \|w(t+\tau)\|_{L^\I} \leq \|e^{(t+\tau)\Delta}w(t)\|_{L^\I} + \frac {c_2} {2\sqrt{t+\tau}},
\]
implying
\[ 
 \clower  \leq \sup_{0<\td \tau<\I} {(t+\td \tau)^{1/2}}  \|e^{(t+\td \tau)\Delta}w(t)\|_{L^\I} \sim \| w(t)\|_{\dot B^{-1}_{\I,\I}}.
\] 
\end{remark}

\bigskip

\section{Discretized   scenario}

We begin by establishing an analogue to Lemmas \ref{lemma.heat} and \ref{lemma.heat.frequency} in the context of the discretized projection operator.
Recall that we consider a fixed lattice of cubes $\{Q_i\}$ with disjoint interiors, volume $h^3$, and whose closures cover $\R^3$. Denote the center of $Q_i$ by $x_i$.  Let  
\[
I_{h,t} u_0 (x) = \sum_{j} \chi_{Q_j}(x)  \sum_{i} \frac 1 {t^{3/2}}e^{-|x_j-x_i|^2/(4t)}  \int_{Q_i} u_0(y)\,dy.
\]
and let 
\[
J_{h} u_0(x) = \sum_j \chi_{Q_j}(x) \frac 1 {|Q_j|} \int_{Q_j} u_0 (y)\,dy.
\] The next lemma should be understood in analogy with Lemmas \ref{lemma.heat} and \ref{lemma.heat.frequency} but where the interpretation of `sparseness' is understood through the length scale  $h$ of the interpolant operator.

\begin{lemma}\label{lemma.HeatDiscrete}  Fix $t>0$, $p\in [1,\I]$ and $\ga >0$. Take $h\lesssim \ga \sqrt t$. If
\[
\| J_hu_0(x)\|_{L^p}\leq \ga/2 \|u_0\|_{L^p},
\]
then 
\[
\| e^{t\Delta}u_0\|_{L^p}  \leq \ga \|u_0\|_{L^p}.
\]
\end{lemma} 

\begin{proof}
Begin by taking $x\in Q_j$.
We have 
\EQ{
| e^{t\Delta} u_0(x)|\leq |e^{t\Delta} u_0(x) - I_{h,t} u_0(x)| + | I_{h,t} u_0(x)|.
}
Expanding the leading term on the right-hand side gives
\EQ{
 |e^{t\Delta} u_0(x) - I_{h,t} u_0(x)| &\lesssim \sum_{i} \frac 1 {t^{3/2}} \int_{Q_i}  e^{-|x_j-x_i|^2/(4t)} \underbrace{\bigg|   e^{-|x-y|^2/(4t)+|x_j-x_i|^2/(4t)} - 1  \bigg|}_{=: F(x,y,x_j,x_i)}       u_0(y) \,dy 
 \\&\lesssim  
 \sum_{i} F(x,y,x_j,x_i)\frac 1 {t^{3/2}}    e^{-|x_j-x_i|^2/(4t)}       \| u_0\|_{L^p(Q_i)}|Q_i|^{1-1/p}.
}
By the mean value theorem and the fact that the Gaussian is Schwartz, it is possible to show that 
\[
\sup_{i,j} \sup_{x\in Q_j,y\in Q_i} F(x,y,x_j,x_i) \lesssim \frac h {\sqrt t}.
\]
Hence
\EQ{
 \|e^{t\Delta} u_0  - I_{h,t} u_0 \|_{L^p(Q_j)} &\lesssim  \frac h {\sqrt t} |Q_j|^{1/p}\sum_{i} \frac 1 {t^{3/2}}    e^{-|x_j-x_i|^2/(4t)}       \| u_0\|_{L^p(Q_i)}|Q_i|^{1-1/p}
 \\& \lesssim  \frac h {\sqrt t} \sum_{i} \frac {h^3} {t^{3/2}}    e^{-|x_j-x_i|^2/(4t)}       \| u_0\|_{L^p(Q_i)},
}
which is a discrete convolution.
We now apply $\ell^p$ to the sequence $\{\|e^{t\Delta} u_0  - I_{h,t} u_0 \|_{L^p(Q_j)}\}$ and use Young's inequality to obtain 
\[
 \|e^{t\Delta} u_0 - I_{h,t} u_0\|_{L^p(\R^3)} \lesssim  \frac h {\sqrt t} \|u_0\|_{L^p(\R^3)}  \sum_{i} \frac {h^3} {t^{3/2}}    e^{-|x_i|^2/(4t)} \lesssim  \frac h {\sqrt t} \|u_0\|_{L^p(\R^3)}.
\]

We also observe that 
\EQ{
|I_{h,t} u_0(x)|&=  \bigg|\sum_{i}\chi_{Q_i}(x) \frac {h^3 } {t^{3/2}} e^{-|x_j-x_i|^2/(4t)} \frac 1 {|Q_i|}\int_{Q_i} u_0(y)\,dy \bigg|  \lesssim |J_hu_0(x)|,
}
and, so,  
\[
\| I_{h,t}u_0\|_{L^p(\R^3)}\lesssim  \| J_h u_0 \|_{L^p(\R^3)}.
\]

Combining the above observations and taking $h\lesssim \ga \sqrt t$, we obtain
\[
\| e^{t\Delta}u_0 \|_{L^p(\R^3)} \lesssim  \frac h {\sqrt t}\|u_0\|_{L^p(\R^3)} +  \| J_h u_0 \|_{L^p(\R^3)}\leq \ga \|u_0\|_{L^p(\R^3)}.
\]
\end{proof}

\begin{proof}[Proof of Theorem \ref{thrm.mainDiscrete}]
We include details for $3<p<\I$.  
Our starting point is 
\EQ{
\| w(t)\|_{L^p} \leq C_B \int_0^t \frac 1 {(t-s)^{1/2}} \big( \| v w \|_{L^p}  +\|u  w\|_{L^p}   \big) \,ds,
}
where we take $t<\delta$.
We only consider the case $\|v w\|_{L^p} $ as the treatment of $u$ is identical. 

Recall from the proof of Theorem \ref{thrm.main} that, choosing $\eta$ to be  
\EQ{ 
\eta =\frac {\e_1} {c_1-\e_1},
}
we have 
\[
\| v w\|_{L^p(|y|\geq \eta^{-1}\sqrt s)}  \leq \frac {\e_1} {\sqrt s} \|w\|_{L^p(|y|\geq \eta^{-1}\sqrt s)}.
\]
We now choose $\bar h=\bar h(s)$ so that $B_{\eta^{-1}\sqrt s}(0) \subset  Q_0$, where $Q_0$ is the cube centered at the origin with edge-lengths $\bar h$---that is, $\bar h\geq 2 \eta^{-1}\sqrt s$. As $\e_1$ was fixed in the proof of Theorem \ref{thrm.main}, we replace $\e_1$ with $\e_3$ and will subsequently adjust its value.

We next write
\[
\| v w\|_{L^p(|y|< \eta^{-1}\sqrt s)} \leq \| v J_{\bar h}w\|_{L^p(|y|< \eta^{-1}\sqrt s)} +\| v (w-J_{\bar h}w)\|_{L^p(|y|< \eta^{-1}\sqrt s)}.
\]
For the first term on the right-hand side we have   
\[
\| v w\|_{L^p(|y|< \eta^{-1}\sqrt s)} \leq \bigg\| v 	\frac 1 {\bar h^3}\int_{Q_0} w(y)\,dy	\bigg\|_{L^p(Q_0)} \leq c_1 c_4 s^{3/(2p) - 1/2} s^{1/4} \bar h^{-3/2},
\]
where, because $u$ and $v$ are $L^{3,\I}$-weak solutions, we have $\|w\|_{L^2}(t)\leq c_4 t^{1/4}$, for a constant $c_4$ depending on $\|u_0\|_{L^{3,\I}}$.
We further restrict $\bar h$  so that 
\[
 s^{3/(2p) - 1/2} s^{1/4} \bar h^{-3/2}\leq \frac {\e_3} {\sqrt s} \|w\|_{L^p}(s).
\]
For the second term on the right-hand side we have 
\[
\| v (w-J_{\bar h}w)\|_{L^p}  \leq \frac {c_1 \e_3} {\sqrt s} \|w\|_{L^p}(s),
\]
by assumption. 

Combining these bounds we obtain
\EQ{
\| w(t)\|_{L^p} &\lesssim_{c_1,c_4} \e_3 \int_0^t \frac 1 {(t-s)^{1/2}} \frac 1 {s^{1-3/(2p)}} {s^{1/2-3/(2p)}} \|w(s)\|_{L^p}  \,ds
\\&\leq \frac 1 2 t^{3/(2p)-1/2} \sup_{0<s<t}{s^{1/2-3/(2p)}}  \|w(s)\|_{L^p},
}
provided $\e_3$ is chosen small compared to $c_1$ and $c_4$ and
\[
\bar h = \max\bigg\{ 2 \frac {c_1-\e_3} {\e_3} \sqrt s, \bigg( {t^{3/4}}  \frac { t^{3/(2p)-1/2} } {\e_3 \|w\|_{L^p(t)}  }   \bigg)^{2/3}\bigg\}.
\]
This is enough to conclude $w\equiv 0$.

As we have seen in the proofs of Theorems \ref{thrm.main} and \ref{thrm.mainFrequancy}, the case $p=\I$ follows similarly but we need to initialize our argument with the estimate
\EQ{
\| w(t)\|_{L^\I} \leq C_B \int_0^t \frac 1 {(t-s)^{1/2 + 3/(2q)}} \big( \| v w \|_{L^q}  +\|u  w\|_{L^q}   \big) \,ds,
}
where $q$ can be any value in $(3,\I)$.

(Part 2) 
This proof is essentially identical to the proof of part 2 of Theorem \ref{thrm.mainFrequancy}. The only difference is that we replace  $2^{J_2}$ with $h^{-1}$.
We then see by Lemma \ref{lemma.HeatDiscrete} that, if  
\[  
\| J_h w(t_0)\|_{L^p} \leq \frac \ga 2 \| w(t_0)\|_{L^p},
\]
then 
\[
\|e^{(t-t_0)\Delta}w(t_0)\|_{L^p} \leq \frac {c_2} {4t^{1/2-3/(2p)}}.
\]
This implies that we \textit{cannot} have 
\[
\| J_h w(t_0)\|_{L^p} \leq \frac \ga 2 \| w(t_0)\|_{L^p}.
\]

\end{proof}

\section{Conditional predictability criteria}

\begin{proof}[Proof of Theorem \ref{thrm.conditionalPredCrit}]  We have 
\[
\partial_t \| w\|_{L^2}^2 + 2  \| \nb w\|_2^2 \leq 2\int w\cdot \nb u w\,dx.
\]
Note that 
\[
\| \Delta_{>J} w \|_2\leq C 2^{-J} \|\nb w\|_{2},
\]
which is just a Bernstein inequality.
Hence,
\[
\partial_t \| w\|_{L^2}^2 +C2^{2J} \| \Delta_{>J}w\|_{2}^2 +     \| \nb w\|_2^2 \leq C \| u\|_{L^\I} \|\nb w\|_{2 }\|w\|_2 \leq C^2 \|u\|_{L^\I}^2 \| w\|_2^2 + \|\nb w\|_{2 }^2.
\]
We alternatively have
\[
\partial_t \| w\|_{L^2}^2 +C2^{2J} \| \Delta_{>J}w\|_{2}^2 +     \| \nb w\|_2^2  \leq C \|\nb u\|_{L^\I} \| w\|_2^2.
\]
Assume that at some  time (which everything that follows occurs at) and some $K>0$ that
\[
\|\Delta_{\leq  J} w\|_{L^2}^2 \leq K \| \Delta_{>J}w \|_{L^2}^2.
\]
Then,
\[
\| w\|_2^2 \leq \| w_{\leq J} \|_2^2 + \|w_{\geq J} w\|_2^2 \leq  (K+1) \|w_{\geq J} \|_2^2.
\]
And, provided 
\[
C 2^{2J} \| \Delta_{>J}w\|_{2}^2 \geq  2(K+1)C^2 \|u\|_{L^\I}^2  \|\Delta_{>J} w\|_2^2,
\]
or
\[
C 2^{2J} \| \Delta_{>J}w\|_{2}^2 \geq  2(K+1)C \|\nb u\|_{L^\I}  \|\Delta_{>J} w\|_2^2,
\]
which is implied if
$(K+1)^{1/2} \min \{\|u\|_{L^\I} , \sqrt{\|\nb u\|_{L^\I}}	\}\lesssim 2^{J}$,
we obtain 
\[
\partial_t \| w\|_{L^2}^2+ \frac C {2} 2^{2J} \| \Delta_{>J}w\|_{2}^2  \leq 0.
\]
Noting again that  $\| w\|_2^2\leq  (K+1) \|w_{\geq J} \|_2^2$ we improve this to 
\[
\partial_t \| w\|_{L^2}^2+ \frac {C2^{2J}}{(K+1)} \|  w\|_{2}^2  \leq 0.
\]

We now think of $J$ as fixed and define $K$   according to $(K+1)^{1/2} \min \{\|u\|_{L^\I} , \sqrt{\|\nb u\|_{L^\I}}	\}= c_5 2^{J}$ where $c_5$ is a universal constant consistent with the suppressed constant above. In this case we see that, if 
\[
\frac{ \|\Delta_{\leq  J} w\|_{L^2}^2 } {\| \Delta_{>J}w \|_{L^2}^2 }\leq  \bigg(  \frac {c_5 2^{2J}}  { \min \{\|u\|_{L^\I} , \sqrt{\|\nb u\|_{L^\I}}\}^2} -1 \bigg),
\]
then 
\[
\partial_t \|w\|_{L^2}^2  + \frac {C \min \{\|u\|_{L^\I} , \sqrt{\|\nb u\|_{L^\I}}\}^2} {2c_5} \|w\|_{L^2}^2  \leq \partial_t \| w\|_{L^2}^2+ \frac C {2} 2^{2J} \| \Delta_{>J}w\|_{2}^2  \leq 0,
\]
and the error is non-increasing.
The preceding condition makes sense provided $K\geq 0$, which implies 
\[
2^{2J} \geq \frac {\min \{\|u\|_{L^\I} , \sqrt{\|\nb u\|_{L^\I}}\}^2} {c_5}.
\]

\bigskip The proof for the discretized operator is identical once we observe that, in place of Bernstein's inequality, we have by the Poincar\'e inequality that 
\[
\| w - J_h w\|_{L^2} \leq C h \|\nb w\|_{L^2}.
\]
Then, we just replace $2^J$ with $h^{-1}$, $\Delta_{\leq J} w$ with $J_h w$ and $\Delta_{>J} w$ with $w-J_h w$  throughout the proof.

 \end{proof}

\begin{remark} One can use sparseness to get a similar result but it seems sub-optimal compared to the energy methods employed above. In particular, it is possible to prove that for $T>0$ given, there exist $J$ and $h$ so that $2^{-J} \sim h \sim M^{-1}$ where $M= \sup_{0<s<T}(\| u\|_{L^\I} + \| v\|_{L^\I})(s)$ and,
  if 
  \[
   \| w_{\leq J} \|_{L^2}\leq \frac 1 4 \|w \|_{L^2}
    \text{ or }
    \|J_h w\|_{L^2}\leq \frac 1 4 \|w \|_{L^2},
    \]
    for all $t\in (0,T)$, then 
    \[
    \sup_{0<s<T} \|w(s)\|_{L^2}\leq 2 \| w(0)\|_{L^2}.
    \]
To execute the sparseness argument using   \eqref{def.wMILD}, we need $M$ to depend on $\|u\|_{\I}$ \textit{and} $\|v\|_{\I}$. This is not the case if we use energy methods   due to the standard cancellation. This is why only a single quantity in Theorem \ref{thrm.conditionalPredCrit} needs to be finite. Furthermore, the decay rate in the proof of Theorem \ref{thrm.conditionalPredCrit} does not follow obviously from the sparseness argument.  So, even though mild solution methods  allow us to suppress the $L^2$ norm, they seem sub-optimal compared to energy methods.

\end{remark}

\section{The energy separation  of Jia and \v Sver\' ak's  hypothetical  solutions}

\begin{proof}[Proof of Proposition \ref{prop.JS}]
We consider the solutions in \cite{JiaSverakIll} which are guaranteed to exist under what Jia and \v Sver\' ak label spectral condition (B)---see \cite[Theorem 5.2]{JiaSverakIll}. The proof of \cite[Theorem 5.2]{JiaSverakIll} is not written explicitly, but follows the same logic as the proof of \cite[Theorem 5.1]{JiaSverakIll}. In particular, \cite[Theorem 3.1]{JiaSverakIll} is applied to non-unique solutions with $O(|x|^{-1})$ data.  We will label these solutions as $u_i$ for $i=1,2$. The difference between the proof of \cite[Theorem 5.1]{JiaSverakIll} and the proof of \cite[Theorem 5.2]{JiaSverakIll} is that $\phi=0$ for both $u_i$ in the latter---this is because under spectral condition (B), both $u_i$ are self-similar. In the statement of \cite[Theorem 3.1]{JiaSverakIll}, this means that $\td a$ is a self-similar local energy solutions to \eqref{eq.ns} while $\td b=0$. For each $u_i$, the application of \cite[Theorem 3.1]{JiaSverakIll} produces a perturbation $\td u_i$ so that $u_i+\td u_i$ is another solution to \eqref{eq.ns}. The initial perturbation is in $L^4$ and, roughly speaking, cuts off the tail of $u_i$ so that $u_i +\td u_i$ has better decay that $u_i$. Our observation is that, as $t\to 0$, the perturbation part vanishes and $u_i+\td u_i$ becomes close to $u_i$ near the origin. This means that, near the origin, the difference between $u_1+\td u_1$ and $u_2+\td u_2$ is close to that of $u_1-u_2$, which is self-similar. 

The main ingredients in our argument are \cite[(3.32) \& (3.33)]{JiaSverakIll},
which state
\[
\lim_{t\to 0^+}t^{1/2}\|\nb \td u_i  \|_{L^4} = 0 \text{ and }  \lim_{t\to 0^+} \| \td u_i - u_0 \|_{L^4} =0,
\]
where $i=1,2$.
These will allow us to show that, for any $\e$, there exists $T_\e$ so that 
\EQ{\label{target}
\| (u_i+\td u_i)-u_i\|_{L^2(B(0,1)}^2(t)\leq \epsilon t^{1/2},
}
for $t\in (0,T_\e]$. Noting that for $t\leq 1$
\[
t^{1/2} \| u_1 - u_2 \|^2_{L^2(B(0,1))}(1)  =\| u_1 - u_2 \|^2_{L^2(B(0,t^{1/2}))}(t)  \leq   \| u_1 - u_2 \|^2_{L^2(B(0,1))}(t). 
\]
we get the result by, e.g., taking $\sqrt \epsilon = \| u_1 - u_2 \|_{L^2(B(0,1))}(1)  / 8$ as this implies
\EQ{
\|(u_1+\td u_1) - (u_2 + \td u_2)\|_{L^2(B(0,1))} &\geq \| u_1 - u_2 \|_{L^2(B(0,1))}- \|\td u_1\|_{L^2(B(0,1))} - \|\td u_2\|_{L^2(B(0,1))} 
\\&\geq t^{1/4} \| u_1 - u_2 \|_{L^2(B(0,1))}(1)  - t^{1/4}\| u_1 - u_2 \|_{L^2(B(0,1))}(1)  / 4.
}

Importantly, we assume that the initial errors are supported off of $B(0,2)$. 
To show \eqref{target}, we   use the local energy inequality \eqref{ineq:CKN-LEI}.  Let $\phi\geq 0$ belong to $C_c^\I(\R^3)$ evaluate to $1$ on $B(0,1)$ and have  support in $B(0,2)$. Note that for small enough $t$ we have 
\[
 \int_0^t \int \td p_i \td u_i \cdot \nb \phi \,dx\,dt \lesssim C\int_0^t \int_{B(0,2)\setminus B(0,1)} \frac 1 {(|x|+\sqrt s)^2}  |\td u_i|\,dx\,ds \lesssim  t \|\td u_i\|_{L^4} \leq \frac \e  4 t^{1/2},
\]
since $t=o(t^{1/2})$.
Other terms in which derivatives fall on $\phi$ can be similarly bounded, where $t$ is restricted so their sum is less than $\e t^{1/2}/2$. The delicate term is the critical order drift term
\[
\int_0^t \int (\td u_i\cdot \nb \td u_i ) \td a  \phi \,dx\,ds.
\]
To bound this note that $\td a$ is a local energy solution to \eqref{eq.ns} (indeed, $\td a = u_i$) and so satisfies \eqref{ineq.apriorilocal}.  Consequently,
\[
\int_0^t \int (\td u_i\cdot \nb \td u_i )\td a   \phi \,dx\,ds  \lesssim_{\td a} \int_0^t \| \td u_i\|_{L^4}\| \nb \td u_i\|_{L^4} \,ds.
\]
In view of \cite[(3.32) \& (3.33)]{JiaSverakIll}, by taking $t$ small we obtain 
\[
\int_0^t \int (\td u_i\cdot \nb \td u_i ) \td a   \phi \,dx\,ds  \leq \frac 1 4 \epsilon \int_0^t \frac 1 {s^{1/2}} \,ds = \frac 1 2 \epsilon t^{1/2}.
\]
Noting that the initial errors are zero on the support of $\phi$ and considering these bounds in the context of the  local energy inequality, we obtain
\[
\int_{B(0,1)} |\td u_i|^2(x,t) \,dx  \leq \e t^{1/2},
\]
for sufficiently small $t$.
\end{proof}

\section*{Acknowledgements}

The research of Z.~Bradshaw was supported in part by the  NSF via grant DMS-2307097.

\bigskip 
\noindent Zachary Bradshaw, Department of Mathematical Sciences, 309 SCEN,
University of Arkansas,
Fayetteville, AR 72701. \url{zb002@uark.edu}
 
\end{document}